\documentclass[10pt,a4paper]{amsart}
\usepackage{amscd,amsfonts,amsmath,amsthm,amssymb,verbatim,mathrsfs}
\usepackage[dvips]{graphicx}
\usepackage{graphics,hyperref}
\usepackage{enumitem}
\usepackage{psfrag}
\usepackage[T1]{fontenc}
\usepackage{lmodern,pst-node, xfrac}
\usepackage{latexsym}
\usepackage{pstcol,pst-plot,pst-3d}
\usepackage{multicol}
\usepackage[normalem]{ulem}

\def\NZQ{\Bbb}               
\def\NN{{\NZQ N}}
\def\QQ{{\NZQ Q}}
\def\ZZ{{\NZQ Z}}

\def\PP{{\NZQ P}}

\def\frk{\frak}               

\def\Phi{{\frk n}}
\def\Phi{{\frk N}}
\def\ub{{\bold u}}
\def\vb{{\bold v}}
\def\ab{{\bold a}}
\def\bb{{\bold b}}

\def\cb{{\bold c}}

\def\opn#1#2{\def#1{\operatorname{#2}}} 
\opn\chara{char} \opn\length{\ell} \opn\pd{pd} \opn\rk{rk}
\opn\projdim{proj\,dim} \opn\injdim{inj\,dim} \opn\rank{rank}
\opn\depth{depth} \opn\grade{grade} \opn\height{height}
\opn\embdim{emb\,dim} \opn\codim{codim} \opn\sgn{sgn}  \opn\GLM{GLM}

\opn\Tr{Tr} \opn\bigrank{big\,rank}
\opn\superheight{superheight}\opn\lcm{lcm}
\opn\trdeg{tr\,deg}
\opn\reg{reg} \opn\lreg{lreg} \opn\ini{in} \opn\lpd{lpd}
\opn\size{size}\opn\bigsize{bigsize}
\opn\cosize{cosize}\opn\bigcosize{bigcosize}
\opn\sdepth{sdepth}\opn\sreg{sreg}
\opn\link{link}\opn\fdepth{fdepth}
\opn\div{div} \opn\Div{Div} \opn\cl{cl} \opn\Cl{Cl} \opn\Cor{Cor}
\opn\Spec{Spec} \opn\Supp{Supp} \opn\supp{supp} \opn\Sing{Sing}
\opn\Ass{Ass} \opn\Min{Min}\opn\Mon{Mon} \opn\dstab{dstab} \opn\astab{astab}
\opn\Ann{Ann} \opn\Rad{Rad} \opn\Soc{Soc} \opn\Gr{Gr}
\opn\Im{Im} \opn\Ker{Ker} \opn\Coker{Coker} \opn\Am{Am}
\opn\Hom{Hom} \opn\Tor{Tor} \opn\Ext{Ext} \opn\End{End}
\opn\Aut{Aut} \opn\id{id} \opn\span{span}

\opn\nat{nat}
\opn\pff{pf}
\opn\Pf{Pf} \opn\GL{GL} \opn\SL{SL} \opn\mod{mod} \opn\ord{ord}
\opn\Gin{Gin} \opn\Hilb{Hilb}\opn\sort{sort}
\opn\aff{aff} \opn\con{conv} \opn\relint{relint} \opn\st{st}
\opn\lk{lk} \opn\cn{cn} \opn\core{core} \opn\vol{vol}
\opn\link{link} \opn\star{star}\opn\lex{lex} \opn\Gr{Gr}
\opn\gr{gr}

\def\pot#1#2{#1[\kern-0.28ex[#2]\kern-0.28ex]}

\opn\dirlim{\underrightarrow{\lim}}
\opn\inivlim{\underleftarrow{\lim}}
\def\Implies{\ifmmode\Longrightarrow \else
        \unskip${}\Longrightarrow{}$\ignorespaces\fi}
\def\implies{\ifmmode\Rightarrow \else
        \unskip${}\Rightarrow{}$\ignorespaces\fi}
\def\iff{\ifmmode\Longleftrightarrow \else
        \unskip${}\Longleftrightarrow{}$\ignorespaces\fi}

\let\:=\colon
\newtheorem{Theorem}{Theorem}[section]
\newtheorem{thm}[Theorem]{Theorem}
\newtheorem{lem}[Theorem]{Lemma}

\newtheorem{ex}[Theorem]{Example}

\newtheorem{defi}[Theorem]{Definition}

\begin{document}

\title{ Self-dual toric varieties }

\author{Apostolos Thoma, Marius Vladoiu}
\thanks{}
\address{Apostolos Thoma, Department of Mathematics, University of Ioannina, Ioannina 45110, Greece} 
\email{athoma@uoi.gr}
\address{Marius Vladoiu, Faculty of Mathematics and Computer Science, University of Bucharest, Str. Academiei 14, Bucharest, RO-010014, Romania}
\email{vladoiu@fmi.unibuc.ro}
\thanks{Corresponding author: Apostolos Thoma}

\begin{abstract}  
We describe explicitly all multisets of weights whose defining projective toric varieties are self-dual. 
In addition, we describe a remarkable and unexpected combinatorial behaviour of the defining ideals of these varieties. The toric ideal 
of a self-dual projective variety is weakly robust, that means the Graver basis is the union of 
all minimal binomial generating sets.  When, in addition, the self-dual projective variety has a non-pyramidal configuration, then the toric ideal is strongly robust, namely
 the Graver basis is a minimal generating set,
therefore there is only one minimal binomial generating set which is also a reduced Gr\"obner basis with respect to every monomial
order and thus, equals the universal Gr\"obner basis.
\end{abstract}
\maketitle

\section{Introduction}

Let $K$ be an algebraically closed field of characteristic zero and $ A$ be an $m\times n$ integer matrix such
that $(1,1,\cdots, 1,1)$ belongs in the row span of the matrix $A$. For a vector 
${\bf u}\in\Ker_{\ZZ}(A)$, we denote 
by ${\bf u}^+, {\bf u}^-\in \NN^n$, the unique
vectors  of disjoint support such that ${\bf u}={\bf u}^+-{\bf u}^-$. The support of ${\bf u}$ is the set $\{i | u_i\neq 0\}$. The toric ideal of $A$ is the 
homogeneous ideal $I_A\subset K[x_1,\ldots,x_n]$ generated by the binomials ${x}^{{\ub}^+}-{x}^{{\ub}^-}$ 
 where ${\ub}\in\Ker_{\ZZ}(A),$ see \cite[Chapter 4]{St}. Let $T=(K^*)^m$ be the algebraic torus of rank $m$ over $K$  and $\chi (T)$ be the lattice of characters of $T$;  characters are of the form $\lambda ({\bf t}) = {\bf t}^{\bf b} = t_1^{b_1}t_2^{b_2}\cdots t_m^{b_m}$, where ${\bf b}\in \ZZ^m$, thus we can identify  $\chi (T)$ with $\ZZ^m$.
 Then $T$ acts on $K^n$ by ${\bf t}\cdot (x_1, x_2, \ldots ,x_n)=({\bf t}^{{\bf a}_1}x_1, {\bf t}^{{\bf a}_2}x_2, \ldots , {\bf t}^{{\bf a}_n}x_n)$, where the set of weights ${\bf a}_1, {\bf a}_2, \ldots ,{\bf a}_n $ are the columns of the matrix $A$. The projective toric variety $X_A=V(I_A)\subset \PP_{K}^{n-1}$ is the Zariski closure of the orbit of the point $(1,1,\ldots, 1,1)$.
By $X_A^*$ we denote the dual variety of $X_A$,  which is the closure  in the dual projective space  of the hyperplanes intersecting the regular points of $X_A$ non transversally. The variety $X_A$ is called self-dual if it is isomorphic to its dual $X_A^*$ as embedded projective varieties. 
In \cite{BDR},  Bourel, Dickenstein and Rittatore prove that the variety $X_A$ is a self-dual projective toric variety
if and only if $\dim(X_A)=\dim(X_A^*)$ and the smallest linear subspaces containing $X_A$ and $X_A^*$ have the same dimension. The variety 
$X_A$ is called defective if the codimension of its dual variety is not one. Thus self-dual toric varieties which are not hypersurfaces are defective. The classification of defective projective toric varieties is a difficult problem with a lot of recent interest, see \cite{D1, D2, D3, D4, D5}.

Also, in \cite{BDR}, Bourel et al. produce the full list of self-dual smooth projective toric varieties and several examples of non-smooth self dual toric vatieties. They also provide combinatorial characterizations of self-duality for toric varieties. 
In  this article, we combine the combinatorial characterizations of \cite{BDR} and the full bouquet theory, including the generalized Lawrence matrices \cite{PTV1}, which permit us to describe explicitly all multisets of weights whose defining projective toric varieties are self-dual and we prove their interesting algebro-combinatorial
properties.
Toric ideals of self-dual projective toric varieties associated to non-pyramidal configurations posess
the same combinatorial properties \cite[Proposition 7.1]{St} as the toric ideals of the second Lawrence liftings: they are {\em strongly robust},  see \cite{Sul}. A toric ideal is called robust if it is minimally generated by its universal Gr{\"o}bner basis, where the  universal Gr{\"o}bner basis is the union of all reduced   Gr{\"o}bner bases, see \cite{BR}.   Robust related properties of ideals are studied in several recent articles; see \cite{BZ, B, cng, CNG, SZ} for robust ideals, \cite{BR,  BBDLMNS} for robust toric ideals,
\cite{    GP, KTV, KTV2, PTV1, PTV2, St,  Sul} for strongly robust toric ideals and \cite{G-MT, T} for generalized robust toric ideals.

  A {\it strongly robust} toric ideal  is a toric ideal $I_A$ for which  the Graver basis $\Gr(I_A)$
is a minimal system of generators, see \cite{Sul}.  Since $(1,1,\cdots, 1,1)$ belongs to the row span of the matrix $A$, then $\Ker_{\ZZ}(A)\cap \NN^n=\{{\mathbf{0}}\}$, which implies that any minimal binomial generating set is contained in the Graver basis, see \cite[Theorem 2.3]{CTVAC}. Then, for strongly robust ideals,  the Graver basis is the unique minimal system of generators and, thus, any reduced Gr{\"o}bner basis as well as the universal Gr{\"o}bner basis are identical with the Graver basis, since all of them contain a minimal system of generators and they are  subsets of the Graver basis (see \cite[Chapter 4]{St}). We conclude that for a strongly robust toric ideal $I_A$ the 
following sets are identical: the set of indispensable elements, any minimal system of binomial generators, any reduced Gr{\"o}bner basis, the universal Gr{\"o}bner basis and the Graver basis. Therefore, strongly robust toric ideals are robust. The set of indispensable elements is the intersection of all minimal binomial sets of generators, while their union is the universal Markov basis, up to the identification of opposite binomials. {\em Weakly robust} are the toric ideals for which the universal Markov basis and the Graver basis of $I_A$ coincide. The ideals of self-dual toric varieties, pyramidal or non-pyramidal, are weakly robust.

In Section 2, we present the basic theory of bouquets.  In Section 3, we make an  introduction to the generalized Lawrence matrices and we use them to
describe explicitly all multisets of weights whose defining projective toric varieties are self-dual. Finally, in Section 4, we present the necessary notions from the different toric bases and 
 we prove that the toric ideal 
of a self-dual projective variety is weakly robust. When, in addition, the self-dual projective variety is defined by  a non-pyramidal configuration, then the toric ideal is strongly robust.

\section{Preliminaries on bouquets} \label{Section 2}
  
Let $\{{\bf a}_1,\dots, {\bf a}_n\}\subset \ZZ^m$.  Throughout the article we will use the same symbol, $A$, for the set  $\{{\bf a}_1,\dots, {\bf a}_n\}$ and for the integer matrix $[{\bf a}_1,\dots, {\bf a}_n]\in\ZZ^{m\times n}$  with columns ${\bf a}_1,\dots, {\bf a}_n$.

 To any  integer matrix $A\in\ZZ^{m\times n}$ 
one can associate an oriented matroid structure consisting of a graph $G_A$ and another matrix $A_B$, called the  ``bouquet graph'' and the ``bouquet matrix'', respectively. The name bouquet was inspired by the article \cite{PStasi} and the theory of bouquets was developed in a series of articles  \cite{PTV1, PTV2, KTV}. But a lot of the ideas behind the bouquet structure were used extensively before the definition   to study the self-dual projective toric varieties in \cite{BDR}.

    The {\em bouquet graph} $G_A$ of $I_A$ is the graph on the set of vertices $\{{\bf a}_1,\dots, {\bf a}_n\}$ and  edges $\{\ab_i,\ab_j\}$  if there exists a vector $\vb\in\ZZ^m$ such that  the dot products $\vb\cdot \ab_k=0$ for all $k\neq i,j$ and $\vb\cdot\ab_i\neq  0$, $\vb\cdot\ab_j\neq 0$ or equivalently that there is a covector with support $i,j$ in the vector matroid defined by the columns of $A$, see \cite[Definition 1.1]{PTV1}.  The {\it bouquets} are the connected components of the graph $G_A$. 

The bouquet graph  can also be  defined in terms of the Gale transform of the matrix $A$.  To any matrix $A$ one can associate its Gale transform, which is the $n\times (n-r)$ matrix whose columns span the lattice $\Ker_{\ZZ}(A)$, where $r$ is the rank of $A$. We will denote the set of ordered row vectors of the Gale transform of $A$ by $\{G(\ab_1), \dots, G(\ab_n)\}$ and we will call  $G(\ab_i)$ the Gale transform of  $\ab_i$. There is an edge between the vertices $\ab_i$ and $\ab_j$ if and only if  $\ab_i$, $\ab_j$ are  coparallel, that means $G(\ab_i), G(\ab_j)$ are parallel i.e.  there exists a nonzero $\lambda\in\QQ$ such that $G(\ab_i)=\lambda G(\ab_j)$.   

The vector ${\bf a}_i$ is called {\em free} if $G(\ab_i)$ is equal to the zero vector, which means that $i$ is not contained in the support of any element in $\Ker_{\mathbb{Z}}(A)$. It follows from the definition that the free vectors of $A$ form one bouquet, which we call the {\em free bouquet} of $G_A$.  If the configuration of weights ${\bf a}_1, {\bf a}_2, \cdots ,{\bf a}_n $ has no free vectors then it is called non-pyramidal.
If the configuration has $k$ free vectors is called $k$-pyramidal. 
 The non-free bouquets are of two types: {\em mixed} and {\em non-mixed}. A non-free bouquet is mixed if contains an edge $\{\ab_i,\ab_j\}$ such that $G(\ab_i)=\lambda G(\ab_j)$ for some $\lambda<0$, and is non-mixed if it is either an isolated vertex or for all of its edges $\{\ab_i,\ab_j\}$ we have $G(\ab_i)=\lambda G(\ab_j)$ with $\lambda>0$, see \cite[Lemma 1.2]{PTV1}.  

 Let $B_1,\ldots,B_s$ be the bouquets of $A$. Then we may rearrange the column vectors in such a way  that there exist positive integers $i_1, i_2, \ldots, i_s$ such that the first $i_1$ column vectors $\ab_1,\ldots,\ab_{i_1}$ belong to the first bouquet $B_1$, the next $i_2$ column vectors belong to the second bouquet $B_2$, and so on the last $i_s$ vectors belong to the $s$-th bouquet $B_s$, thus in particular $i_1+\cdots+i_s=n$.

Next, we recall from \cite{PTV1}, the definition of the \emph{bouquet-index-encoding vector} $\cb_B$.
If the bouquet $B$ is free then we set $\cb_B\in\ZZ^n$ to be any nonzero vector such that
$\supp(\cb_B)=\{i| \ab_i\in B\}$ and with the property that the first nonzero component is positive.
For a  non-free bouquet $B$ of $A$, consider the Gale transforms of the elements in $B$.
 All the elements are nonzero and pairwise linearly dependent, therefore there exists at least one nonzero component $j$ in all of them. Let $g_j=\gcd(G({\bf a}_i)_j| \ {\bf a}_i\in B)$ and fix the smallest
integer $\ell$ such that ${\bf a}_{\ell}\in B$. Let ${\bf c}_B$  be the vector in $\ZZ^n$ whose $i$-th component
is $0$ if ${\bf a}_i\notin B$, and is $\varepsilon_{\ell j}G({\bf a}_i)_j/g_j$ if ${\bf a}_i \in B$, where   $G(\ab_i)_j$ represents the $j$-th component of the vector $G(\ab_i)$ and $\varepsilon_{\ell j}$  the sign of the integer $G({\bf a}_{\ell})_j$. Note that the choice of $\ell$ implies that the first nonzero component of ${\bf c}_B$ is positive. Division by $g_j$ makes the gcd of the components of ${\bf c}_B$ to be one.
 Since each ${\bf a}_i$ belongs to  exactly one bouquet the supports of the vectors ${\bf c}_{B_i}$ are pairwise disjoint. In addition, $\cup_i\supp(\cb_{B_i})=[n]$, and taking into account that the first $i_1$ vectors are in $B_1$, the next $i_2$ in $B_2$, and so forth, then $\cb_{B_1}$ has only the first $i_1$ components nonzero, $\cb_{B_2}$ has only the following $i_2$ components nonzero, and so forth. Therefore, we may label the nonzero components of $\cb_{B_1},\ldots,\cb_{B_s}$ as follows:  the nonzero components of $\cb_{B_q}$ are $c_{q1},\ldots,c_{qi_q}$, for all $1\leq q\leq s$. 

The structure of a non-free bouquet $B$ of $A$, that is mixed or non-mixed, can be detected from the vector $\cb_B$. If $B_i$ is a  non-free bouquet of $A$, then $B_i$ is a mixed bouquet if and only if the vector ${\bf c}_{B_i}$ has a negative and a positive component, and  $B_i$ is non-mixed if and only if the vector $\cb_{B_i}$ has all nonzero components positive, see \cite[Lemma 1.6]{PTV1}.

For each bouquet $B$ we define the following vector: 
 $$\ab_B=\sum_{j=1}^n (c_B)_j\ab_j\in\ZZ^m,$$
where $(c_B)_j$ is the $j$-th component of the vector $\cb_B$. Note that the ${\bf a}_j$'s that contribute in fact to the vector ${\bf a}_B$ are those belonging to the bouquet $B$. The matrix $A_B$ whose column vectors are the vectors $\ab_{B_1},\ab_{B_2},\ldots, \ab_{B_s}$,
corresponding to the bouquets of $A$, is called the {\em bouquet matrix} of $A$.

Finally, we recall (see \cite[Theorem 1.9]{PTV1}) that the matrix $A_B\in\ZZ^{m\times s}$ has its kernel $\Ker_{\ZZ}(A_B)$ isomorphic to $\Ker_{\ZZ}(A)$ via the isomorphism $D:\Ker_{\ZZ}(A_B)\mapsto\Ker_{\ZZ}(A)$ given by
\[
D(u_1,\ldots,u_q)=(c_{11}u_1,\ldots,c_{1i_1}u_1,c_{21}u_2,\ldots,c_{2i_2}u_2,\ldots,c_{s1}u_s,\ldots,c_{si_s}u_s).
\]
Note that the vector $D({\bf u})\in \Ker_{\ZZ}(A)$ has $n$ components, which we will denote by  $u_{jt}$ in the rest of the paper, where $1\leq j \leq s$, $1\leq t\leq i_j$.

The following theorem describes the connection between the toric ideals of $A$ and  $A_B$. By  $\mathcal C(A)$
we denote the {\em circuits} of $A$, which consists of all  nonzero vectors ${\bf{u}}\in\Ker_{\ZZ}(A)$ with minimal support and the gcd of their components equal to one. While by  $\Gr(A)$ we denote the {\em Graver basis} of  $A$, which consists of all  nonzero vectors ${\bf{u}}\in\Ker_{\ZZ}(A)$ with no proper conformal decomposition,  see \cite[Section 4]{St}. We say that ${\bf u}={\bf v}+_{c} {\bf w}$ is a conformal decomposition of the vector ${\bf u}$ in two vectors {\bf{v}}, {\bf{w}}, if ${\bf u}={\bf v}+{\bf w}$ and ${{\bf u}}^+={{\bf v}}^++{{\bf w}}^+, {{\bf u}}^-={{\bf v}}^-+{{\bf w}}^-$. We call the decomposition proper if both ${\bf{ v}}$ and ${\bf{ w}}$ are nonzero. 

\begin{thm}[{\cite[Theorems 1.9 and 1.11]{PTV1}}]\label{all_is_well}
Let $A=[{\bf a}_1,\dots,{\bf a}_n]\in\ZZ^{d\times n}$ and its bouquet matrix $A_B=[\ab_{B_1},\dots,\ab_{B_s}]$. There is a bijective correspondence between the elements of $\Ker_{\ZZ}(A_B)$ $(\Gr(A_B)$ and $\mathcal C(A_B)$, respectively) and the elements of $\Ker_{\ZZ}(A)$ $(\Gr(A)$ and $\mathcal C(A)$, respectively). More precisely, this correspondence is defined as follows: if ${\bf u}=(u_1,\ldots,u_{s})\in\Ker_{\ZZ}(A_B)$, then $D({\bf u})={\bf c}_{B_1}u_1+\cdots+{\bf c}_{B_s}u_s\in\Ker_{\ZZ}(A)$.
\end{thm}

\begin{lem}\label{non_pyramid_pres_bouq}
Let $A=\{\ab_1,\ldots,\ab_n\}$ be a  configuration and $A_B=\{\ab_{B_1},\ldots,\ab_{B_s}\}$ be its bouquet configuration. Then $A$ is non-pyramidal if and only if $A_B$ is non-pyramidal.
\end{lem}
\begin{proof} Suppose that $A$ is a non-pyramidal configuration and $A_B$ is pyramidal. Then there exists an ${{\bf a}_B}_j$ which is a free vector, thus $G({{\bf a}_B}_j)={\bf 0}$. From the definition of the Gale transform this means that every
element ${\bf u}$ in the  $\Ker_{\ZZ}(A_B)$ has the $j$-th component equal to $0$. Then, according to 
Theorem~\ref{all_is_well}, all elements $D({\bf u})$ in the $\Ker_{\ZZ}(A)$ have their components $u_{j1},\cdots, u_{ji_j}$ corresponding to the $B_j$ bouquet zero, since $u_{jt}=c_{jt}u_j=0$. Therefore all the Gale transforms
of the vectors of the bouquet $B_j$ are zero and thus the vectors are free, contradicting the fact that $A$ is non-pyramidal configuration. 

For the other direction, suppose that $A_B$ is a non-pyramidal configuration and $A$ is a pyramidal. Then there exists a free vector in $A$, thus its Gale transform is zero. Suppose that this free vector is the $l$-th one in the bouquet $B_j$. Then all
elements $D({\bf u})$ in the  $\Ker_{\ZZ}(A)$ have their $jl$-th component equal to $0$. According to 
Theorem~\ref{all_is_well}, the $jl$-th component is $c_{jl}u_j$. But $c_{jl}\not =0$ thus $u_j=0$, for all ${\bf u}\in \Ker_{\ZZ}(A_B)$. Therefore,  the Gale transform
of the vector ${{\bf a}_B}_{j}$ is zero and thus the vector  ${{\bf a}_B}_{j}$ is free, contradicting the fact that $A_B$ is a non-pyramidal configuration. \end{proof}

\section{Generalized  Lawrence matrices and Classification results}

Next we recall the definition of a generalized Lawrence matrix from \cite[Section 2]{PTV1}. 
Let $\{\ab'_1,\ldots,\ab'_s\}\subset\ZZ^d$ be a configuration and $\cb'_1,\ldots,\cb'_s$ be any set of primitive vectors, a {\em primitive} vector is an integer vector such that the gcd of
all its components is $1$, with $\cb'_i\in\ZZ^{m_i}$ for some $m_i\geq 1$, each having full support and a positive first component. We denote a generalized Lawrence matrix by $\GLM(\ab'_1,\ldots,\ab'_s|\cb'_1,\ldots,\cb'_s)$ and is defined as follows: for each $i=1,\dots,s$, let  $\cb'_i=(c'_{i1},\ldots,c'_{im_i})\in\ZZ^{m_i}$ and define 
\[
C_i=
\left( \begin{array}{ccccc}
 -c'_{i2} &  c'_{i1} &  0 & \ldots & 0  \\
 -c'_{i3} & 0 & c'_{i1} & \ldots & 0  \\
\vdots & \vdots & \vdots &  \ddots & \vdots \\
 -c'_{im_i} & 0 & 0 & \ldots  & c'_{i1}
\end{array} \right)
	 \in\ZZ^{(m_i-1)\times m_i}.
\]
	Primitivity of each $\cb'_i$ implies that there exist integers $\lambda_{i1},\ldots,\lambda_{im_i}$ 
	such that  $1=\lambda_{i1}c'_{i1}+\cdots+\lambda_{im_i}c'_{im_i}$. Fix a choice of $\lambda_{i1},\ldots,\lambda_{im_i}$, 
	and define the matrices $A_i=[\lambda_{i1}{\bf a}'_i,\ldots,\lambda_{im_i}{\bf a}'_i]\in\ZZ^{d\times m_i}$. 
	The generalized Lawrence matrix is then the following block matrix: 
\[
\GLM(\ab'_1,\ldots,\ab'_s|\cb'_1,\ldots,\cb'_s)=\
\left( \begin{array}{cccc}
 A_1 & \ A_2 & \cdots & \ A_s \\
 C_1 &\ 0\ & \cdots & 0 \\
 0 &\ C_2\ & \cdots & 0 \\
\vdots&\vdots & \ddots & \vdots \\
 0 &\ 0 & \cdots &\ C_s
\end{array} \right)  \in \ZZ^{p\times n}, 
\]
where $p=d+(m_1-1)+\cdots+(m_s-1)$ and $n=m_1+\cdots+m_s$, see \ref{example1} for an example of a generalized Lawrence matrix. 
Note that the definition of a generalized Lawrence matrix depends also on the choice of $\lambda_{ij}$, and thus there are infinitely many choices for $\GLM(\ab'_1,\ldots,\ab'_s|\cb'_1,\ldots,\cb'_s)$. However, all of them have the same kernel and thus  define the same toric variety,
see \cite[Section 2]{PTV1}.

\begin{thm}[{\cite[Theorem 2.1]{PTV1}}]\label{inverse_construction}
Let $\{\ab'_1,\ldots,\ab'_s\}\subset\ZZ^d$ be an arbitrary set of vectors. 
Let $\cb'_1,\ldots,\cb'_s$ be any set of primitive vectors, with $\cb'_i\in\ZZ^{m_i}$ for some $m_i\geq 1$, each having full support
and a positive first component. Then, any generalized Lawrence matrix  $A=\GLM(\ab'_1,\ldots,\ab'_s|\cb'_1,\ldots,\cb'_s)\in\ZZ^{p\times n}$ 
has a subbouquet decomposition, $B_1,\ldots,B_s$, such that the $i^{th}$ subbouquet is encoded by the following vectors:
$\ab_{B_i}=(\ab'_i,{\bf 0},\ldots,{\bf 0})\in\ZZ^{p}$ and $\cb_{B_i}=({\bf 0},\ldots,\cb'_i,\ldots,{\bf 0})\in\ZZ^{q}$, where $n=m_1+\cdots+m_s$, $p=d+q-s$ and
the support of $\cb_{B_i}$ is precisely  the $i^{th}$ block of $\ZZ^n$ of size $m_i$. 
\end{thm}

\begin{lem}\label{projective_GLM}
Let $A=\GLM(\ab'_1,\ldots,\ab'_s|\cb'_1,\ldots,\cb'_s)$ be such that $\sum_j c'_{ij}=0$ for all $i=1,\ldots,s$. Then $X_A$ is a projective toric variety. 
\end{lem}
\begin{proof}
 Adding up the rows of $C_i$ we obtain the row vector $(-c'_{i2}-\cdots-c'_{im_i},c'_{i1},\ldots,c'_{i1})$ which is equal by the hypothesis with $(c'_{i1},c'_{i1},\ldots,c'_{i1})$, 
 and since $c'_{i1}\neq 0$ we conclude that  ${\bf 1}_i=(1,\ldots,1)\in \ZZ^{m_i}$ lies in the rowspan of $C_i$. Since this can be done for every $i$ by adding the corresponding row vectors  $({\bf 0},\ldots,{\bf 1}_i,\ldots,{\bf 0})\in\ZZ^{n}$ of $A$ we obtain
 that  $(1,\ldots,1)\in \ZZ^{n}$ lies in the rowspan of $A$.
 \end{proof}
Therefore the projectiveness of $X_A$, when $A=\GLM(\ab'_1,\ldots,\ab'_s|\cb'_1,\ldots,\cb'_s)$,  does not depend on the vectors $\ab'_i$, in other words on the bouquet ideal, but only on the property of the bouquet-index-encoding vectors that $\sum_j c'_{ij}=0$ for all $i=1,\ldots,s$. For example in \ref{example1}  the bouquet ideal is the ideal of an affine monomial curve.

The next theorem is a reformulation of {\cite[Theorem 4.4]{BDR}} since the  vectors that their Gale transforms are parallel form a bouquet. 

\begin{thm} \label{3.3}
Let $A=[\ab_1\cdots\ab_n]$ be a non pyramidal configuration. Then $X_A$ is self-dual if and only if $\sum_{_{\ab_i\in B}}G(\ab_i)={\bf 0}$ for any bouquet $B$ of $A$. 
\end{thm}

\begin{thm} \label{non_pyramid}
Let $A=[\ab_1\cdots\ab_n]$  be a non pyramidal configuration. Then $X_A$ is self-dual if and only if $\sum_i (c_B)_i=0$ for any bouquet $B$ of $A$. 
\end{thm}

\begin{proof} The configuration is non pyramidal therefore all bouquets are non-free. 
  We recall the definition of the vector ${\bf c}_B$. Let $j$ be such that all
   $j$-components of the Gale transforms of the elements in $B$ are nonzero. Let $g_j=\gcd(G({\bf a}_i)_j| \ {\bf a}_i\in B)$ and fix the smallest
integer $\ell$ such that ${\bf a}_{\ell}\in B$.  Then $(c_B)_i=\varepsilon_{\ell j}G(\ab_i)_j/g_j$,
  where $\varepsilon_{\ell j}$ represents the sign of the integer $G({\bf a}_{\ell})_j$. From Theorem \ref{3.3} we have  $\sum_{_{\ab_i\in B}}G(\ab_i)={\bf 0}$ for any bouquet $B$ of $A$. Then the sum of their $j$-components is $0$, thus  $\sum_i G(\ab_i)_j=0$ therefore $\sum_i \varepsilon_{\ell j}G(\ab_i)_j/g_j=0$. We conclude that $\sum_i (c_B)_i=0$ for any bouquet $B$ of $A$.
\end{proof}

\begin{thm}\label{self_dual_GLM_nonpyramid}
Let $A=\GLM(\ab'_1,\ldots,\ab'_s|\cb'_1,\ldots,\cb'_s)$ be a generalized Lawrence matrix with $\ab'_1,\ldots,\ab'_s$
a non-pyramidal configuration. 
Then $X_A$ is a self-dual projective toric variety if and only if $\sum_j c'_{ij}=0$ for all $i=1,\ldots,s$.  
\end{thm}
\begin{proof}
From  \cite[Theorem 2.1, Corollary 2.3]{PTV1}  we know that $A$ has exactly $s$ bouquets and the encoding bouquet vectors of $A$ are $\ab_{B_i}=(\ab'_i,{\bf 0},\ldots,{\bf 0})\in\ZZ^{p}$ and $\cb_{B_i}=({\bf 0},\ldots,\cb'_i,\ldots,{\bf 0})\in\ZZ^{n}$ for $i=1,\ldots,s$. Applying Lemma~\ref{non_pyramid_pres_bouq} for $A$ we obtain that $A$ is a non-pyramidal configuration. The conclusion follows now from Theorem~\ref{non_pyramid}. 
\end{proof}

Theorem  \ref{self_dual_GLM_nonpyramid} gives a way of producing examples of non-pyramidal self-dual projective toric varieties. Take any matrix with columns $\ab'_1,\ldots,\ab'_s$ to be a non-pyramidal configuration, then take a set of $s$ primitive vectors $\cb'_1,\ldots,\cb'_s$ with  the sum of their components zero. Then the toric variety defined by the matrix $\GLM(\ab'_1,\ldots,\ab'_s|\cb'_1,\ldots,\cb'_s)$ is a non-pyramidal self-dual projective toric variety. 
 \begin{ex} \label{example1} {\em
  Let $A$ be any configuration, for example the one given by the columns of the matrix $A=(4 \ 5 \ 6 \ 7)$. Take a primitive vector for each column such that
the sum of the components for each primitive vector is zero, say $c'_1=(1, -1), c'_2=(7,-4,-3), c'_3=(2023, 1, -2024), c'_4=(1, 3, -2, -2)$.
Computing any integer solution for each $1=\lambda_{i1}c'_{i1}+\cdots+\lambda_{im_i}c'_{im_i}$, 
for example $1=1\cdot 1+0\cdot(-1)$, $1=1\cdot 7+0\cdot(-4)+2\cdot(-3)$, $1=0\cdot 2023+1\cdot 1+0\cdot(-2024)$ and $1=0\cdot 1+1\cdot 3+1\cdot(-2)+0\cdot(-2)$ we get the
 generalized Lawrence matrix $C  \in \ZZ^{9\times 12}$:
\[
\tiny{
\GLM(\ab'_1,\ldots,\ab'_4|\cb'_1,\ldots,\cb'_4)=\
\left( \begin{array}{cccccccccccc}
 4 & 0 & 5 & 0 & 10 & 0 & 6 & 0 & 0 & 7 & 7 & 0  \\
 1 & 1 & 0 & 0 & 0 & 0 & 0 & 0 & 0 & 0 & 0 & 0  \\
 0 & 0 & 4 & 7 & 0 & 0 & 0 & 0 & 0 & 0 & 0 & 0  \\
 0 & 0 & 3 & 0 & 7 & 0 & 0 & 0 & 0 & 0 & 0 & 0  \\
 0 & 0 & 0 & 0 & 0 & -1 & 2023 & 0 & 0 & 0 & 0 & 0  \\
 0 & 0 & 0 & 0 & 0 & 2024 & 0 & 2023 & 0 & 0 & 0 & 0  \\
 0 & 0 & 0 & 0 & 0 & 0 & 0 & 0 & -3 & 1 & 0 & 0  \\
 0 & 0 & 0 & 0 & 0 & 0 & 0 & 0 & 2 & 0 & 1 & 0  \\
 0 & 0 & 0 & 0 & 0 & 0 & 0 & 0 & 2 & 0 & 0 & 1  
\end{array} \right). }
\]
Theorem \ref{self_dual_GLM_nonpyramid} proves that $X_C$ is a self-dual projective toric variety.}
\end{ex}
Next Theorem \ref{co} actually says that {\em all} non-pyramidal configurations are taken in the way described above.

\begin{thm}\label{co}
Let $A$ be a non-pyramidal configuration such that $X_A$ is self-dual. Then there exists a generalized Lawrence matrix $A'=\GLM(\ab'_1,\ldots,\ab'_s|\cb'_1,\ldots,\cb'_s)$ such that $X_A=X_{A'}$ after a possible renumeration of the columns, where    $\ab'_1,\ldots,\ab'_s\in\ZZ^d$ is a non-pyramidal configuration and $\sum_j c'_{ij}=0$ for all $i=1,\ldots,s$.  
\end{thm}
\begin{proof}  For the integer matrix $A$ we know from \cite[Corollary 2.3]{PTV1} that there exists a generalized Lawrence matrix $A'=\GLM(\ab'_1,\ldots,\ab'_s|\cb'_1,\ldots,\cb'_s)$
    such that $I_{A}=I_{A'}$, up to permutation of column indices. Then $A'$ is also a non-pyramidal configuration, therefore by Lemma \ref{non_pyramid_pres_bouq} and Theorem \ref{inverse_construction} the
    configuration $\ab'_1,\ldots,\ab'_s\in\ZZ^d$ is non-pyramidal. The last condition that $\sum_j c'_{ij}=0$ for all $i=1,\ldots,s$ follows from Theorem \ref{self_dual_GLM_nonpyramid}.
\end{proof}

For the next result we need one notation: the vector $\varepsilon_i\in\ZZ^d$, $1\leq i\leq d$, represents the integer vector with the only nonzero component equal to $1$ and on the $i$-th position.  

\begin{thm}\label{self_dual_GLM_pyramid}
Consider the generalized Lawrence matrix $$C=\GLM(\varepsilon_1,\ab'_1,\ldots,\ab'_s|\cb'_0,\cb'_1,\ldots,\cb'_s)$$ with $n$ columns, $\cb'_0\in\ZZ^k$, $c'_{01}=1$, $k\geq 1$, $\ab'_i=(0,*,\ldots,*)$ and 
$\sum_j c'_{ij}=0$ for all $i=1,\ldots,s$. If the configuration $\{\ab'_1,\ldots,\ab'_s\}\subset \ZZ^d$ is non-pyramidal, then for any $n+k$ multiset configuration $A$ with the ground set the set of columns of $C$  we have that $X_A$ is a self-dual projective toric variety. 
\end{thm}
\begin{proof}
First we note that, by the proof of Lemma \ref{projective_GLM}, the vector $(0,\ldots,0,1,\ldots,1)\in \ZZ^{n}$ lies in the rowspan of $C$, where the first $k$ components are zero and the rest are one.
The  matrix \[
\left(\begin{array}{c}  A_0 \\ C_0 \end{array} \right) =
\left( \begin{array}{ccccc}
\lambda_{01} & \lambda_{02} & \lambda_{03} & \ldots & \lambda_{0k} \\
 -c'_{02} &  c'_{01} &  0 & \ldots & 0  \\
 -c'_{03} & 0 & c'_{01} & \ldots & 0  \\
\vdots & \vdots & \vdots &  \ddots & \vdots \\
 -c'_{0k} & 0 & 0 & \ldots  & c'_{01}
\end{array} \right)
=\left( \begin{array}{ccccc}
1 & 0 & 0 & \ldots & 0 \\
 -c'_{02} &  1 &  0 & \ldots & 0  \\
 -c'_{03} & 0 & 1 & \ldots & 0  \\
\vdots & \vdots & \vdots &  \ddots & \vdots \\
 -c'_{0k} & 0 & 0 & \ldots  & 1
\end{array} \right)
	 \in\ZZ^{k\times k},
\]
since we can choose as the solution to $1=\lambda_{01}c'_{01}+\cdots+\lambda_{0k}c'_{0k}$ with $c'_{01}=1$ the
$\lambda_{01}=1$ and  $\lambda_{0j}=0$, for all $2\leq j\leq k$. Therefore in the row span of the matrix 
$\left(\begin{array}{c}  A_0 \\ C_0 \end{array} \right)$ belongs every vector of $\ZZ^{k}$,
thus also the $(1,\ldots,1)\in \ZZ^{k}$. Note that the first row of $C$ and the  $d+1$-th up to $(d+k)$-th rows have the last $n-k$ components equal to 0. Concluding, we see that  $(1,\ldots,1,0,\ldots,0)\in \ZZ^{n}$ lies in the row span of $C$, where the first $k$ components are one and the rest zero.
 Combining the two results up to this point we get that $(1,\ldots,1,1,\ldots,1)\in \ZZ^{n}$ belongs to the row span of $C$ and so
the toric variety $X_C$ is projective. $A$ is a  multiset configuration  with  ground set the set of columns of $C$, therefore the toric variety $X_A$ is also projective. 

Theorem 3.8 of \cite{BDR} states that $X_A$ is self dual if and only if (1) the configuration $C$ is $k$-pyramidal
and (2) there exists a splitting of the lattice of characters $\ZZ^p$
as $S_1\times S_2$ such that, after reordering of the elements of $C$, it holds that $C={\mathcal C}_1\cup {\mathcal C}_2$, where 
${\mathcal C}_1$ is a basis of $\chi (S_1)$ and ${\mathcal C}_2\subset \chi (S_2)$ is a non-pyramidal configuration and the toric variety $X_{{\mathcal C}_2}$ is self-dual. 

We denote the vector columns of the matrix $C\in\ZZ^{p\times n}$ by $\ab_1,\ldots,\ab_n$. The vectors in ${\mathcal C}_1=\{\ab_1,\ldots,\ab_k\}$ are free vectors  since there exist vectors $\ub_i\in\QQ^p$ with $i=1,\ldots,k$ such that $$(\ub_i\cdot\ab_1,\ldots,\ub_i\cdot\ab_n)=\varepsilon_{i}\in\ZZ^p.$$
Indeed, the vector $\ub_1=\varepsilon_1$ and the vectors $\ub_i=c'_{0i}\varepsilon_1+\varepsilon_{d+i-1}$ for all $i=2,\ldots,k$ satisfy the above conditions, since 

\[  C=
\left( \begin{array}{ccccc}
 \epsilon_1 &  0 & \  0 & \cdots & \  0 \\
  0 & A_1 &\ A_2 & \cdots & \ A_s \\
 C_0 &  0 &\ 0 & \cdots & \  0 \\
  0 & C_1 &\ 0 & \cdots & \ 0 \\
  0 & 0 &\ C_2 & \cdots & \ 0 \\
\vdots & \vdots & \vdots & \ddots & \vdots \\
  0 & 0 &\ 0 & \cdots &\ C_s
\end{array} \right)  \in \ZZ^{p\times n}, 
\]
where the vector $\varepsilon_1\in\ZZ^k$. Thus $\ab_1,\ldots,\ab_k$ are free vectors, therefore the condition
(1) of \cite[Theorem 3.8]{BDR} is satisfied. 
Let 
$$\chi(S_1)=\{(v_1, 0, 0, \cdots, 0, v_{d+1}, \cdots, v_{d+k}, 0, \cdots, 0)|v_i\in \ZZ\},$$  
$$\chi(S_2)=\{(0, v_2, v_3, \cdots,  v_{d}, 0, \cdots, 0, v_{d+k+1},  \cdots, v_p)|v_i\in \ZZ\},$$
then $\ZZ^p= \chi(S_1)\times \chi(S_2)$. So ${\mathcal C}_1$ is a basis of $\chi (S_1)$.
Let ${\mathcal C}_2=\{ \ab_{k+1},\ldots,\ab_{n}\}$ be the remaining column vectors of $C$. Denote also by ${\mathcal C}_2$ the matrix with columns $ \ab_{k+1},\ldots,\ab_{n}$. The two matrices ${\mathcal C}_2$ and $\GLM(\ab'_1,\ldots,\ab'_s|\cb'_1,\ldots,\cb'_s)$ have the same kernel, since the two matrices differ by some rows of zeroes. Thus they define the same toric ideal and variety. 

The configuration $\{\ab'_1,\ldots, \ab'_s\}$ is non-pyramidal.
It follows  from Lemma \ref{non_pyramid_pres_bouq} and Theorem \ref{inverse_construction} that the columns of the generalized Lawrence matrix $$\GLM(\ab'_1,\ldots,\ab'_s|\cb'_1,\ldots,\cb'_s)$$ form a non-pyramidal configuration. Thus ${\mathcal C}_2\subset \chi (S_2)$ is a non-pyramidal configuration.  The toric variety $X_{{\mathcal C}_2}$ is self-dual by Theorem \ref{self_dual_GLM_nonpyramid}, since $\sum_j c'_{ij}=0$ for all $i=1,\ldots,s$. 

Both conditions of \cite[Theorem 3.8]{BDR} are satisfied thus $X_A$ is a self-dual projective toric variety. \end{proof}

Theorem  \ref{self_dual_GLM_pyramid} gives a way of producing examples of pyramidal or non-pyramidal self-dual projective toric varieties. Take any matrix with columns $\ab'_1,\ldots,\ab'_s$, with  $\ab'_i=(0,*,\ldots,*)$, to be a non-pyramidal configuration, then take a set of $s$ primitive vectors $\cb'_1,\ldots,\cb'_s$ with  the sum of their components zero. Then create the $p\times n$-matrix  $C=\GLM(\varepsilon_1,\ab'_1,\ldots,\ab'_s|\cb'_0,\cb'_1,\ldots,\cb'_s)$ with
$\cb'_0\in\ZZ^k$, $c'_{01}=1$ and $k\geq 1$. Take any $n+k$ multiset configuration $A$ with the ground set the set of columns of $C$, then $X_A$ is a self-dual projective toric variety. Theorem \ref{10} determines under what conditions on the multiset $A$ the resulting self-dual toric variety will be pyramidal or not. Next Theorem \ref{converse_self_dual_GLM_pyramid} says that {\em all} self-dual projective toric varieties are produced in this way.

\begin{thm}\label{converse_self_dual_GLM_pyramid}
Let $A$ be a multiset configuration  with the ground set $C$ such that $X_A$ is self-dual. Then there exists a generalized Lawrence matrix $$C'=\GLM(\varepsilon_1,\ab'_1,\ldots,\ab'_s|\cb'_0,\cb'_1,\ldots,\cb'_s)$$ and a multiset configuration $A'$ with ground set the set of columns of $C'$, such that $X_A=X_{A'}$ after a possible renumeration of the columns, where $k=|A|-|C|$, $\cb'_0\in\ZZ^k$, $c'_{01}=1$,   $\ab'_i=(0,*,\ldots,*)\in\ZZ^d$, $\ab'_1,\ldots,\ab'_s$ a non-pyramidal configuration and $\sum_j c'_{ij}=0$ for all $i=1,\ldots,s$.  
\end{thm}
\begin{proof}  Let $A$
 be the  multiset configuration $\{\underbrace{\ab_1,\ldots,\ab_1}_{k_1+1},\ldots,\underbrace{\ab_n,\ldots,\ab_n}_{k_n+1}\}$ with the ground set 
$C=\{\ab_1,\ldots,\ab_n\}$
such that $X_A$ is self-dual and  $k=\sum_{i=1}^n k_i$. 
Then from \cite[Theorem 3.8]{BDR} the configuration $C$ is $k$-pyramidal
and  there exists a splitting of the lattice of characters $\ZZ^p$
as $S_1\times S_2$ such that, after reordering of the elements of $C$, it holds that $C={\mathcal C}_1\cup {\mathcal C}_2$, where 
${\mathcal C}_1$ is a basis of $\chi (S_1)$ and ${\mathcal C}_2\subset \chi (S_2)$ is a non-pyramidal configuration and the toric variety $X_{{\mathcal C}_2}$ is self-dual. $\chi (S_1)\simeq \ZZ^k$ and we can suppose that ${\mathcal C}_1=\{\varepsilon_1, \varepsilon_2, \cdots ,\varepsilon_k\}\subset \chi(S_1)=\{(v_1, v_2,  \cdots, v_k, 0,  \cdots, 0)|v_i\in \ZZ\}$. Let $ \chi(S_2)=\{(0, \cdots, 0, v_{k+1}, v_{k+2},  \cdots, v_p)|v_i\in \ZZ\}$.
  For the integer matrix ${\mathcal C}_2$ we know from \cite[Corollary 2.3]{PTV1} that there exists a generalized Lawrence matrix ${\mathcal C}_2'=\GLM(\ab'_1,\ldots,\ab'_s|\cb'_1,\ldots,\cb'_s)$
    such that $I_{{\mathcal C}_2}=I_{{\mathcal C}'_2}$, up to a permutation of column indices. Therefore ${\mathcal C}_2$ and ${\mathcal C}'_2$ have the same integer kernel, up to a permutation of the components, as well as the matrices
     $C=\left( \begin{array}{cc}
 I_k & {\bf 0} \\
 {\bf 0} & {\mathcal C}_2 
\end{array} \right)$ and   
    $\left( \begin{array}{cc}
 I_k & {\bf 0} \\
 {\bf 0} & {\mathcal C}_2' 
\end{array} \right)$. The second matrix after the row operations $R_i\leftrightarrow R_{d+i}$, $2\leq i\leq k$, and then  $R_i\rightarrow R_i-c'_{0i}R_1$, $2\leq i\leq k$,  which do not change the kernel can be brought to the form $C'=\GLM(\varepsilon_1,\ab'_1,\ldots,\ab'_s|\cb'_0,\cb'_1,\ldots,\cb'_s)$. By Lemma \ref{non_pyramid_pres_bouq} the configuration    $\ab'_1,\ldots,\ab'_s$ is non-pyramidal and by Theorem 
\ref{self_dual_GLM_nonpyramid} the $\sum_j c'_{ij}=0$ for all $i=1,\ldots,s$. Let $\sigma$ be the permutation such that $\Ker_{\ZZ}(C)=\Ker_{\ZZ}(\sigma(C'))$.
Let $\bb_1,\ldots,\bb_n$ be the columns of $C'$ and $A'$ be the multiset configuration $\{\underbrace{\bb_{\sigma (1)},\ldots,\bb_{\sigma (1)}}_{k_1+1},\ldots,\underbrace{\bb_{\sigma (n)},\ldots,\bb_{\sigma (n)}}_{k_n+1}\}$ then $\Ker_{\ZZ}(C)=\Ker_{\ZZ}(\sigma(C'))$ implies  $\Ker_{\ZZ}(A)=\Ker_{\ZZ}(A')$. Thus $X_A=X_{A'}$.
\end{proof}
\begin{ex}\label{example2}
 {\em
Let $A_1$ be a simple 1-pyramidal configuration, for example the one given by the columns of the matrix $\left( \begin{array}{ccccc}
 1 & 0 & 0 & 0 & 0\\
 0 & 4 & 5 & 6 & 7
\end{array} \right)$.

Take a primitive vector for each column such that the sum of the component for each primitive vector is zero, except possible for the first 
one  say $c'_0=(1, 3, 5, 7), c'_1=(1, -1), c'_2=(7,-4,-3), c'_3=(2023, 1, -2024), c'_4=(1, 3, -2, -2)$. Computing any integer solution for each
$1=\lambda_{i1}c'_{i1}+\cdots+\lambda_{im_i}c'_{im_i}$ (for $i=0$ a solution is $1=1\cdot 1+0\cdot 3+0\cdot 5+0\cdot 7$) 
we get a generalized Lawrence matrix $\GLM(\varepsilon_1, \ab'_1,\ldots,\ab'_4|\cb'_0,\ldots,\cb'_4)$ as follows
\[
C_1=\tiny{
\left( \begin{array}{cccccccccccccccc}
1 & 0 & 0 & 0 & 0 & 0 & 0 & 0 & 0 & 0 & 0 & 0 & 0 & 0 & 0 & 0 \\
0 & 0 & 0 & 0 & 4 & 0 & 5 & 0 & 10 & 0 & 6 & 0 & 0 & 7 & 7 & 0  \\
-3 & 1 & 0 & 0 & 0 & 0 & 0 & 0 & 0 & 0 & 0 & 0 & 0 & 0 & 0 & 0 \\
-5 & 0 & 1 & 0 & 0 & 0 & 0 & 0 & 0 & 0 & 0 & 0 & 0 & 0 & 0 & 0 \\
-7 & 0 & 0 & 1 & 0 & 0 & 0 & 0 & 0 & 0 & 0 & 0 & 0 & 0 & 0 & 0 \\
0 & 0 & 0 & 0 & 1 & 1 & 0 & 0 & 0 & 0 & 0 & 0 & 0 & 0 & 0 & 0  \\
0 & 0 & 0 & 0 & 0 & 0 & 4 & 7 & 0 & 0 & 0 & 0 & 0 & 0 & 0 & 0  \\
0 & 0 & 0 & 0 & 0 & 0 & 3 & 0 & 7 & 0 & 0 & 0 & 0 & 0 & 0 & 0  \\
0 & 0 & 0 & 0 & 0 & 0 & 0 & 0 & 0 & -1 & 2023 & 0 & 0 & 0 & 0 & 0  \\
0 & 0 & 0 & 0 & 0 & 0 & 0 & 0 & 0 & 2024 & 0 & 2023 & 0 & 0 & 0 & 0  \\
0 & 0 & 0 & 0 & 0 & 0 & 0 & 0 & 0 & 0 & 0 & 0 & -3 & 1 & 0 & 0  \\
0 & 0 & 0 & 0 & 0 & 0 & 0 & 0 & 0 & 0 & 0 & 0 & 2 & 0 & 1 & 0  \\
0 & 0 & 0 & 0 & 0 & 0 & 0 & 0 & 0 & 0 & 0 & 0 & 2 & 0 & 0 & 1  
\end{array} \right) .}
\]

Let $E$ be any 16+4 multiset configuration with the ground set the set of columns of $C_1$. Theorem \ref{self_dual_GLM_pyramid} proves that $X_E$ is a self-dual projective toric variety.
Theorem \ref{converse_self_dual_GLM_pyramid} says that all multiset configurations such that $X_E$ is a self-dual projective toric variety are taken in this way.
Note that if we repeat once all first four columns the corresponding variety is non-pyramidal, while in all other cases is s-pyramidal, where s is
less than or equal to four and depends on the choices of the four vectors that we repeat, see the examples \ref{Example A}, \ref{Example B}. This is the subject of the next Theorem.}
\end{ex}
\begin{thm} \label{10} Let 
\[
A=\{\underbrace{\ab_1,\ldots,\ab_1}_{k_1+1},\ldots,\underbrace{\ab_n,\ldots,\ab_n}_{k_n+1}\}
\]
a multiset configuration with $k_1,\ldots,k_n$ nonnegative integers with the ground set $C=\{\ab_1,\ldots,\ab_n\}\subset \ZZ^d$,  $\ab_i\neq\ab_j$ for all $i\neq j$ such that $X_A$ is self dual. Let $k=\sum_{i=1}^n k_i$ be the total number of repetitions.  Then $A$ is an $s$-pyramidal configuration, with $0\leq s\leq k$. In particular:\\
(1) $A$ is a $k$-pyramidal configuration if and only if  $k_i=0$ for every ${\bf a}_i$ free vector of the $C$ configuration 
and \\
(2) $A$ is  a non-pyramidal configuration if and only if $k_i=1$ for every ${\bf a}_i$  free vector of the $C$ configuration 
and $k_i=0$ for every ${\bf a}_i$ non-free vector of the $C$ configuration.
\end{thm}
\begin{proof} By \cite[Theorem 3.8]{BDR} $C$ is $k$-pyramidal, which means it has $k$ free vectors. 
We claim that the vector ${\bf a}_i$ is free in $A$ if and only if  the vector ${\bf a}_i$  is free 
in $C$ and $k_i=0$.\\
Suppose that  the vector ${\bf a}_i$ is free in $A$ then there exist ${\bf u}\in \ZZ^d$ such that ${\bf u}\cdot {\bf a}_i=1$ for only one vector of $A$ and ${\bf u}\cdot {\bf a}_j=0$ for all the rest.
Thus $k_i=0$ and ${\bf a}_i$  is free 
in $C$. For the other direction, suppose that
the vector ${\bf a}_i$  is free in $C$, then there exist ${\bf u}\in \ZZ^d$ such that ${\bf u}\cdot {\bf a}_i=1$ and ${\bf u}\cdot {\bf a}_j=0$ for all $j\not =i$. Then ${\bf u}\cdot {\bf a}_i=1$ for only one vector in $A$, since $k_i=0$, and ${\bf u}\cdot {\bf a}_j=0$ for all other vectors in $A$. Thus  ${\bf a}_i$ is free in $A$.

From the previous claim we conclude that the number of free vectors in $A$ is less than or equal to the number of free vectors in $C$, thus $s\leq k$. The  equality $s=k$ happens if and only if $k_i=0$ for every free vector of $C$.

Suppose now that $A$ is 
non-pyramidal configuration therefore $k_i\ge 1$ for all free vectors ${\bf a}_i$. Thus
$$k=\sum_{i=1}^n k_i\ge \sum_{{\bf a}_i \text{free}} k_i\ge k.$$
Therefore $k=\sum_{i=1}^n k_i= \sum_{{\bf a}_i \text{free}} k_i= k$ which means that $k_i=0$ if ${\bf a}_i$ is a non-free vector and $k_i=1$ if ${\bf a}_i$ is a free vector of the $C$ configuration. For the converse, 
suppose that $k_i=1$ if ${\bf a}_i$ is a free vector of $C$ then there is no free vector in $A$ and thus the $A$ configuration is non-pyramidal.
\end{proof}

\section{Toric bases and combinatorial properties of self-dual toric varieties} \label{4}

A \textit{Markov basis} of an $m\times n$ matrix $A$ is a finite subset  $M$ of $\Ker_{\mathbb{Z}}(A)$ such that whenever $\textbf{w}, \textbf{v} \in \mathbb{N}^n$
and $\textbf{w}- \textbf{v} \in \Ker_{\mathbb{Z}}(A)$ there exists a set $\{\textbf{u}_i | i = 1, \ldots , l\}$  that \textit{connects} $\textbf{w}$
to $\textbf{v}$, where  either  $\textbf{u}_i \in M$ or the opposite  $-\textbf{u}_i \in M$. To connect $\textbf{w}$
to $\textbf{v}$ means that $\textbf{w}- \textbf{v} = \sum_{i=1}^{l} \textbf{u}_i$, and $(\textbf{w}-\sum_{i=1}^{p} \textbf{u}_i) \in \mathbb{N}^n$ for all $1 \leq p \leq l$. A Markov basis of $A$ is {\em minimal} if no proper subset of it is a Markov basis of $A$. A renowned result of Diaconis and Sturmfels
\cite[Theorem 3.1]{DS} says that  $\mathcal{M}$ is a minimal Markov basis  of $A$ if and only if the set $\{\textbf{x}^{u^+}-\textbf{x}^{u^-}|\textbf{u}\in \mathcal{M}\} $ forms a minimal system of binomial generators for the toric ideal $I_A$.
The universal Markov basis $\mathcal{M}(A)$ of $A$ is the union of all minimal Markov bases of $A$, up to identification of opposite vectors. The set of indispensable elements $\mathcal{S}(A)$ is the intersection of all different minimal Markov bases, via the same identification.   The set of indispensable binomials  $\mathcal{S}(I_A)$ of the toric ideal is the set  $\{\textbf{x}^{u^+}-\textbf{x}^{u^-}|\textbf{u}\in \mathcal{S}(A) \}$, that is the intersection  of all minimal systems of binomial generators of $I_A$, up to identification of opposite binomials.

We say that ${\bf{u}}={\bf{v}}+_{sc}{\bf{w}}$ is a semiconformal decomposition of the vector {\bf{u}} in two vectors {\bf{v}}, {\bf{w}}, if {\bf{u}}={\bf{v}}+{\bf{w}} and $v_i>0$ implies $w_i\geq 0$ and $w_i<0$ implies that $v_i\leq 0$ for every $1\leq i\leq n$,  a concept introduced first in \cite{HS}, where by $v_i,w_i$ we denote the $i$-th component of the vectors ${\bf v}$ and ${\bf w}$, respectively. We call the decomposition  proper if both ${\bf v}, {\bf w}$ are nonzero.   From \cite[Proposition 1.1]{CTVJA}, we know that the set of the indispensable elements $\mathcal{S}(A)$ of $A$ consists of all nonzero vectors in $\Ker_{\ZZ}(A)$ with no proper semiconformal decomposition.

We say that ${\bf u}={\bf v}+_{ssc} {\bf w}$ is a strongly semiconformal decomposition of the vector $\ub$ in two vectors {\bf{v}}, {\bf{w}}, if ${\bf u}={\bf v}+{\bf w}$, ${\bf u}^+>{\bf v}^+$ and ${\bf u}^->{\bf w}^-$. Recall that for two vectors ${\bf{u}}, {\bf{v}}\in \mathbb{Z}^n$ we say that ${\bf{u}}\geq{\bf{v}}$ if ${\bf{u}}-{\bf{v}}\in\mathbb{N}^n$, where the inequality is strict if ${\bf{u}}\neq{\bf{v}}$.  Note that strongly semiconformal decomposition can be of more than two vectors, see \cite[Section 1]{CTVJA}. The universal Markov basis of $A$, $\mathcal{M}(A)$, consists of all nonzero vectors in $\Ker_{\ZZ}(A)$ with no proper strongly semiconformal decomposition, 
\cite[Proposition 1.4]{CTVJA}.

\begin{defi} A toric ideal $I_A$ is called strongly robust 
if the Graver basis of $I_A$ is equal to a minimal Markov basis of $I_A$. 
\end{defi}  

It follows from the definition, as we have seen in the Introduction, that for strongly robust ideals 
the following sets are identical: the set of indispensable elements, any minimal Markov basis,
the universal Markov basis, any reduced Gr\"obner basis,
the universal Gr\"obner basis and  the Graver basis of $A$.

\begin{thm}\label{four_sets}
Let $A=\{\ab_1,\ldots,\ab_n\}$ be a non-pyramidal configuration such that $X_A$ is a self-dual projective toric variety. Then the toric ideal is strongly robust.
\end{thm}
\begin{proof}
Since $A$ is non-pyramidal configuration then by Theorem~\ref{non_pyramid} we obtain $\sum_i (c_B)_i=0$ for every bouquet $B$ of $A$. 
The latter condition implies that each vector $\cb_B$ has at least one positive and one negative component, 
since by definition $\cb_B$ has the first nonzero component positive. Thus each bouquet $B$ is mixed by Theorem \ref{non_pyramid}, 
and applying \cite[Corollary 4.4]{PTV1} we obtain that the toric ideal is strongly robust. 
\end{proof}

\begin{ex}
{\em Let $A=(4 \ 5 \ 6 \ 7)$ and $C$ be the generalized Lawrence matrix defined in the Example \ref{example1}.
 Then the matrix $A=(4 \ 5 \ 6 \ 7)$ has 29 elements in the Graver basis, those are:  
\[
\Gr(A)=\{ (5, -4, 0,  0),  (1, -2,  1,  0),  (2, -3,  0,  1),  (4, -2, -1,  0),  (3, -1,  0, -1), \]
\[ (1, -1, -1,  1),  
(2,  1, -1, -1), (1,  2,  0, -2),  (3,  0, -2,  0),   (2,  0,  1, -2), (5,  0, -1, -2),\]
\[  (4,  1,  0, -3),  (7,  0,  0, -4),
(0,  1, -2,  1),  (2,  2, -3,  0),  (1,  0, -3,  2),  (1,  3, -2, -1), \]
 \[ (1,  1,  2, -3),  (0,  4, -1, -2), (0,  3,  1, -3),
 (1,  4, -4,  0),  (1,  0,  4, -4),  (1, -5,  0,  3), \]
\[ (0,  5, -3, -1),  (0,  2,  3, -4),  (0,  6, -5,  0),  (0,  1,  5, -5),
(0,  7,  0, -5),  (0,  0,  7, -6)\}.
\]

 Theorem \ref{four_sets} says that $I_{C}$ is strongly robust. The Graver basis by \cite[Theorem 3.7]{PTV1} it has 29 elements all in the form
$D({\bf u})=u_1\gamma_1+u_2\gamma_2+u_3\gamma_3+u_4\gamma_4=(u_1,-u_1,7u_2,-4u_2,-3u_2,2023u_3,u_3,-2024u_3,u_4,3u_4,-2u_4,-2u_4)$, \\ where ${\bf u}=(u_1, u_2, u_3, u_4) \in \Gr(A)$ and $\gamma_1=(1,-1,0,0,0,0,0,0,0,0,0,0), \\ \gamma_2=(0,0,7,-4,-3,0,0,0,0,0,0,0), \ \gamma_3=(0,0,0,0,0,2023,1,-2024,0,0,0,0), \\ \gamma_4=(0,0,0,0,0,0,0,0,1,3,-2,-2)$ are the vectors with support $c'_1, c'_2, c'_3, c'_4$ of the Example \ref{example1}. For example, for the element $(1,3,-2,-1)\in \Gr(A)$ we have the corresponding
Graver basis element $$D(1,  3, -2, -1)=(1, -1, 21, -12, -9, -4046, -2, 4048, -1, -3, 2, 2)\in \Gr(C).$$
This element, $D(1,  3, -2, -1)$, is also indispensable thus belongs to every Markov basis and every reduced 
Gr\"obner basis, as also the other 28 elements in the form $D({\bf u})$. This property of $I_C$ is not true for $I_A$, for example $(1,3,-2,-1)\in \Gr(A)$ is not indispensable, since it has a proper semi-conformal decomposition: $$(1,3,-2,-1)=(1,2,0,-2)+_{sc}(0,1,-2,1).$$}
\end{ex}

Note that if we drop the assumption of non-pyramidal configuration then the conclusion is no longer true. 
However, we still have an important combinatorial information. The toric ideal is {\em weakly robust}.
We recall that the universal Markov basis of $I_A$, $\mathcal M(A)$, is the union of all minimal generating sets of $I_A$ modulo the $\pm$ sign (see \cite[Definition 3.1]{HS}) and it has a useful algebraic characterization described in \cite[Proposition 1.4]{CTVJA}, similar to the conformality for Graver basis. 
\begin{defi} A toric ideal $I_A$ is called weakly robust 
if the Graver basis of $I_A$ is equal to the universal Markov basis of $I_A$. 
\end{defi}  

We  examine the relation between the Graver basis of a multiset configuration $A$ with ground set $C$ and the Graver basis of $C$. Let $C=\{\ab_1,\ldots,\ab_n\}\subset \ZZ^d$ be such that $\ab_i\neq\ab_j$ for all $i\neq j$ and 
\[
A=\{\underbrace{\ab_1,\ldots,\ab_1}_{k_1+1},\ldots,\underbrace{\ab_n,\ldots,\ab_n}_{k_n+1}\}
\]
a multiset configuration with $k_1,\ldots,k_n$ nonnegative integers. Denote by $k$ the nonnegative integer
$\sum_{i=1}^n k_i$. If $k=0$ then $A=C$, thus $\Gr(A)=\Gr(C)$. In the case that $k\geq 1$, consider $I_C\subset K[x_1,\ldots,x_n]$ 
and $$I_A\subset K[y_{11},\ldots,y_{1k_1+1},\ldots,y_{n1},\ldots,y_{nk_n+1}].$$ 
 
 Let $\pi $ be the restriction 
from $\Ker_{\ZZ}(A)$  to $ \Ker_{\ZZ}(C)$ 
of the linear map \[
\pi(u_{11},\ldots,u_{1,k_1+1},\ldots,u_{n1},\ldots,u_{n,k_n+1})=(\sum_{i=1}^{k_1+1}u_{1i},\ldots,\sum_{i=1}^{k_n+1}u_{ni})
\] from $\QQ^{n+k}$ to $ \QQ^n$.

Note that the vectors in $A$ are in n-blocks. All the vectors in the i-th block are equal to each other and different from the elements in the other blocks.
The size of the $i-$th block is $k_i+1$.
Also the components of ${\bf u'}$ in $\Ker_{\ZZ}(A)$ are partitioned in n-blocks. The components in the i-th block are denoted by $u_{i1}, \cdots, u_{ik_i+1}$.

For ${\bf u}\in \Ker_{\ZZ}(C)$ we denote by $\pi^{-1}({\bf u})_+$ all elements ${\bf u'}$ in $\Ker_{\ZZ}(A)$ such that $\pi({\bf u'})={\bf u}$
and all nonzero components in the i-block of ${\bf u'}$ have the {\em same sign}, for every $i$-block,
this means that if $u_i>0$ then $u'_{ij}\ge 0$, if $u_i=0$ then $u'_{ij}=0$ and if $u_i<0$ then $u'_{ij}\le 0$, for all possible $j$.
Let $Q_A$ be the set of all the elements ${\bf c}_{is,it}$ in $\Ker_{\ZZ}(A)$ in the form: all components of ${\bf c}_{is,it}$ are zero 
except an 1 in the $is$ position and a $-1$ in the $it$ position, $1\leq s, t \leq k_i+1.$ 

 \begin{thm}\label{graver} The Graver basis of $I_A$ is $$(\bigcup_{{\bf u}\in Gr(C)} \pi^{-1}({\bf u})_+)\bigcup Q_A.$$
\end{thm}
\begin{proof}
The elements in $Q_A$ are circuits, since they are homogeneous of degree one and have support of cardinality two. Therefore they are
in the Graver basis, see \cite[Proposition 4.11]{St}. 

Let ${\bf u}'\in \pi^{-1}({\bf u})_+$ for some ${\bf u}\in Gr(C)$. Suppose that ${\bf u}'={\bf v}'+_c{\bf w}'$,   $\pi $ is linear therefore
$\pi ({\bf u}')=\pi ({\bf v}')+ \pi ({\bf w}')$.
Note that all non-zero components in the i-block of ${\bf u}$ have the same sign, since ${\bf u}'\in \pi^{-1}({\bf u})_+$. Then the conformality
of the sum ${\bf u}'={\bf v}'+_c{\bf w}'$ implies that all non-zero components in the i-th blocks of ${\bf u}', {\bf v}', {\bf w}'$ have the same sign. This
imply that the  sum $\pi ({\bf u}')=\pi ({\bf v}')+_c \pi ({\bf w}')$  is conformal. But ${\bf u}=\pi ({\bf u}')$ is in the Graver of $C$ therefore either $\pi ({\bf v}')$ or 
$\pi ({\bf w}')$
is zero. But  all the non-zero components in the i-th block have the same sign, for every $i\in [n]$, therefore if  $\pi ({\bf v}')={\bf 0}$ then ${\bf v}'={\bf 0}$
and if  $\pi ({\bf w}')={\bf 0}$ then ${\bf w}'={\bf 0}$. 
 Therefore the sum  ${\bf u}'={\bf v}'+_c{\bf w}'$ is not proper and ${\bf u}' \in Gr(A)$.
Thus $(\bigcup_{{\bf u}\in Gr(C)} \pi^{-1}({\bf u})_+)\bigcup Q_A\subset Gr(A)$. 

Let ${\bf u}'\in Gr(A)$ which is not in $Q_A$. First we claim that all non-zero components in any i-th block of ${\bf u}'$ are having the same sign.
Suppose not, then let $ u'_{is}>0$ and  $u'_{it}<0$ for some $s, t$, then define ${\bf u}''$ to have all components the same as ${\bf u}'$ except
$is$ and $it$ and $ u''_{is}= u'_{is}-1\geq 0$, $ u''_{it}= u'_{it}+1\leq 0$. Then the sum ${\bf u}'={\bf u}''+{\bf c}_{is,it}$ is conformal and both
${\bf u}'',{\bf c}_{is,it}$ are in the kernel of A. Note that ${\bf c}_{is,it}$ has all components zero, except the $is$ which is 1 and the $it$ which is -1.
 The ${\bf u}'$ is not in $Q_A$ thus both ${\bf u}'', {\bf c}_{is,it}$ are different from zero.
 A contradiction, since  ${\bf u}'\in Gr(A)$. Therefore  all non-zero elements in the i-th block of ${\bf u}'$ are having the same sign, for every $i\in [n]$.

Next  we claim that $\pi({\bf u}')\in Gr(C)$. Suppose that $\pi({\bf u}')={\bf u}={\bf v}+_c{\bf w}$, where ${\bf v},{\bf w}$ are in the kernel of $C$. Then look at
the i-th block of ${\bf u}'$ 
we have $\sum u'_{ij}=u_i= v_i+w_i$ and since the sum is conformal all nonzero components are of the same sign. 

Suppose first that $u_i>0$.   In the case  $v_i=u_i=\sum_{j=1}^{k_i+1} u'_{ij}$ we set $v_{ij}=u'_{ij}$ and $w_{ij}=0$. Otherwise 
there exist a $t$ such that $\sum_{j=1}^t u'_{ij}\leq v_i<\sum_{j=1}^{t+1} u'_{ij}$ for $0\leq t\leq k_i+1$, where for $t=0$ we set $\sum_{j=1}^t u'_{ij}=0$. Set $v_{ij}=u'_{ij}$ and $w_{ij}=0$ for $j\leq t$,
$v_{it+1}=v_i-\sum_{j=1}^t u'_{ij}\ge 0$ and $w_{ij}=\sum_{j=1}^{t+1} u'_{ij}-v_i>0$ and $v_{ij}=0$ and $w_{ij}=u'_{ij}$ for $j> t$.

In the case that $u_i=0$ then $v_i=0=w_i$ from the conformality. We set all $v_{ij}=0=w_{ij}.$

In the last case  $u_i<0$.   In the case  $v_i=u_i=\sum_{j=1}^{k_i+1} u'_{ij}$ we set $v_{ij}=u'_{ij}$ and $w_{ij}=0$. Otherwise 
there exist a $t$ such that $\sum_{j=1}^t u'_{ij}\geq v_i>\sum_{j=1}^{t+1} u'_{ij}$ for $0\leq t\leq k_i+1$, where for $t=0$ we set $\sum_{j=1}^t u'_{ij}=0$. Set $v_{ij}=u'_{ij}$ and $w_{ij}=0$ for $j\leq t$,
$v_{it+1}=v_i-\sum_{j=1}^t u'_{ij}\le 0$ and $w_{ij}=\sum_{j=1}^{t+1} u'_{ij}-v_i<0$ and $v_{ij}=0$ and $w_{ij}=u'_{ij}$ for $j> t$.

Let ${\bf v}'$ the vector with components 
$v_{ij}$ and ${\bf w}'$  the vector with components 
$w_{ij}$. 
Then ${\bf u}'=
{\bf v}'+_c{\bf w}'$ which implies that either ${\bf v}'={\bf 0}$ or ${\bf w}'={\bf 0}$. 
But from the definition  ${\bf v}'$ is different from zero, since ${\bf u}'\not = {\bf 0}$. Thus 
${\bf w}'={\bf 0}$. But this can happen only if ${\bf w}={\bf 0}$, 
which means that the conformal sum  $\pi({\bf u}')={\bf v}+_c{\bf w}$
is never proper. Therefore $\pi({\bf u}')\in Gr(C)$.  Set $\pi({\bf u}')={\bf u}$ thus  ${\bf u}'\in \pi^{-1}({\bf u})_+$ for a ${\bf u}\in Gr(C)$.
\end{proof}

\begin{thm}\label{all_good}
Let $A$ be any configuration such that $X_A$ is a self-dual projective toric variety. Then the toric ideal $I_A$ is weakly robust.  
\end{thm}

\begin{proof}

If $A$ is non-pyramidal then we are done by  theorem \ref{four_sets}, since strongly robust implies weakly robust.

Otherwise, by applying \cite[Theorem 3.8]{BDR} since $X_A$
is self-dual, we obtain that the ground set of $A$, $C=\{\ab_{1},\ldots,\ab_n\}$, is $k$-pyramidal and, after a permutation, the vectors in $C_f=\{\ab_1,\ldots,\ab_k\}$ are free  and the configuration $C_{nf}=\{\ab_{k+1},\ldots,\ab_n\}$ is non-pyramidal with the corresponding variety $X_{C_{nf}}$ being self-dual.   

 Since $(1,1,\cdots, 1,1)$ belongs in the row span of the matrix $A$ then $\Ker_{\ZZ}(A)\cap \NN^n=\{{\mathbf{0}}\}$ which implies that $\mathcal M(A)\subset Gr(A)$, see \cite[Theorem 2.3]{CTVAC}.  

Next we prove that $\Gr(A)\subset\mathcal M(A)$.
Note that an element ${\bf u}'$ belongs in the universal Markov basis, $\mathcal M(A)$, 
if and only if the corresponding binomial $y^{{\bf u}'_+}-y^{{\bf u}'_-}$ 
cannot be written as a combination 
of binomials of $I_A$ of strictly smaller $A$-degree, see 
\cite[Section 2]{ChKTh} or 
\cite[Theorem 1.3.2]{DSS}  or \cite[Proposition 1.4]{CTVJA}. The $A$-degree of the binomial $y^{{\bf u}'_+}-y^{{\bf u}'_-}$ is $A{\bf u}'_+(=A{\bf u}'_-).$
Let ${\bf u}'\in\Gr(A)$ such that ${\bf u}'\notin\mathcal M(A)$. Then  we 
have $$y^{{\bf u}'_+}-y^{{\bf u}'_-}=\sum y^{{\bf m}_i}(y^{{{\bf u}_i}_+} -y^{{{\bf u}_i}_-})\in K[y_{11},\ldots,y_{1k_1+1},\ldots,y_{n1},\ldots,y_{nk_n+1}],$$ where $y^{{\bf m}_i}\not= 1$, for some ${\bf u}_i\in \Ker_{\ZZ}(A)$. 
Certainly no element in $Q_A$ can be written in this form, since all are of degree 1. Therefore  ${\bf u}'\in \pi^{-1}({\bf u})_+$ for some ${\bf u}\in Gr(C)$, by Theorem \ref{graver}.
Note that the vectors $\ab_1,\ldots,\ab_k$ are free therefore the first $k$ components are zero for every element 
${\bf v}$ in $\Ker_{\ZZ}(C)$, thus $supp({\bf v})\subset \{k+1,\cdots, n\}$. Let $p$ be the projection 
from $\ZZ^n$ to $\ZZ^{n-k}$ to the last $n-k$ components. Then $p({\bf v})\in \Ker_{\ZZ}(C_{nf})$ 
for every ${\bf v}\in \Ker_{\ZZ}(C)$ and for the ${\bf u}\in Gr(C)$ we have $p({\bf u})\in Gr(C_{nf})$,
since conformal decompositions in $\Ker_{\ZZ}(C_{nf})$ give conformal 
decompositions in $\Ker_{\ZZ}(C)$.

Then setting in the above equation $y_{ij}=x_{i}$ for $k+1\leq i\leq n$, $1\leq j\leq k_i+1$, and $y_{ij}=0$ for $1\leq i\leq k$, $1\leq j\leq k_i+1$, we have 
$$x^{{\bf u}_+}-x^{{\bf u}_-}=x^{p({\bf u})_+}-x^{p({\bf u})_-}=\sum x^{\pi({\bf m}_i)}(x^{p({\pi({{\bf u}_i}_+)})}-x^{p({\pi({{\bf u}_i}_-)})}) $$  in the ring $K[x_{k+1},\ldots,x_n]$, with all $x^{\pi({\bf m}_i)}\not= 1$. 
Therefore $p({\bf u})\notin\mathcal M(C_{nf})$.  Note that the sum with the $x$'s in the right part of the equation above may have less terms than the previous one with the $y$'s, but still has at least one, since the left part is not zero. 

The $X_{C_{nf}}$ is a non-pyramidal self-dual variety thus from Theorem \ref{four_sets} the ideal $I_{C_{nf}}$ is strongly robust, and so also weakly robust, thus $\Gr(C_{nf})=\mathcal M(C_{nf})$. 
But $p({\bf u})\in Gr(C_{nf})$ and  $p({\bf u})\notin\mathcal M(C_{nf})$, a contradiction.

We conclude that $\Gr(A)=\mathcal M(A)$.
\end{proof}
Next Theorem proves that for self-dual toric varieties the strongly robust property is equivalent with the property of being non-pyramidal. 
\begin{thm}\label{iff}
Let $A$ be any configuration such that $X_A$ is a self-dual projective toric variety.  The toric ideal $I_A$  is strongly robust if and only if $A$ is non-pyramidal. 
\end{thm}
\begin{proof} By Theorem \ref{four_sets} if $A$ is non-pyramidal then  $I_A$  is strongly robust. Suppose that $A$ is $s$-pyramidal with $s\ge 1$. Then the gound set $C$ is $k$-pyramidal $k\ge s$, by Theorem \ref{10}. 

Suppose that there exist $i$ such that $k_i>1$, that means the vector ${\bf a}_i$ appears at least three times in $A$.
Then the three circuits ${\bf c}_{i1,i2}$, ${\bf c}_{i1,i3}$, ${\bf c}_{i2,i3}$ by Theorem \ref{graver} belong to the Graver basis of $A$. Note that ${\bf c}_{i1,i3}={\bf c}_{i1,i2}+_{sc}{\bf c}_{i2,i3}$ is a proper semiconformal decomposition of ${\bf c}_{i1,i3}$, see introduction of Section \ref{4} or \cite[Proposition 1.1]{CTVJA}, thus ${\bf c}_{i1,i3}$ is not indispensable. But ${\bf c}_{i1,i3}$ is a circuit and therefore belongs to the Graver basis of $I_A$. For strongly robust ideals every Graver basis element is indispensable. Thus  $I_A$ is not 
strongly robust. 

Therefore all $k_i$ take the values one or zero. Since $A$ is a pyramidal configuration by Theorem \ref{10} there exist an $i$ such that ${\bf a}_i$ is non free in the $C$ configuration and $k_i=1$. Since ${\bf a}_i$ is non free vector in $C$ there
exist a ${\bf u}\in Gr(C)$ such that  the $i$-th component of ${\bf u}$ is different from zero, and thus w.l.o.g. we can assume that it is positive. Consider ${\bf u}'\in \pi^{-1}({\bf u})_+$ with $u'_{i1}=u_i$ and thus $u'_{i2}=0$, and ${\bf v}'\in \pi^{-1}({\bf u})_+$ with $v'_{i1}=u_i-1$ and thus $v'_{i2}=1$ then ${\bf u}'={\bf c}_{i1,i2}+_{sc}{\bf v}'$ is a semiconformal
decomposition of ${\bf u}'$ and it is proper since ${\bf c}_{i1,i2}\not ={\bf 0}$ and ${\bf u}'\in \pi^{-1}({\bf u})_+$ which implies that ${\bf u}'\not ={\bf c}_{i1,i2}$ and ${\bf v}'\not ={\bf 0}$.   Thus ${\bf u}'$  is not indispensable, see \cite[Proposition 1.1]{CTVJA}, but it is in the Graver of $A$, see Theorem \ref{graver}. Therefore  $I_A$ is not strongly robust. 
\end{proof}

\begin{ex} \label{Example A}
{\em Let $C_1$ be the matrix of the Example \ref{example2}. 
Let $E_1$ be the matrix in which we repeat once all first four columns of $C_1$ then Theorem \ref{10} says that the columns of $E_1$ form a non-pyramidal configuration. From Theorem \ref{graver} we have that there
are 29+4 elements in the Markov basis which is also universal Gr\"obner and Graver,
as was expected from Theorem \ref{four_sets} since it is non-pyramidal self dual.
The 29 elements have the first eight components zero and the rest are the components of the 29 elements of $I_{C}$, in the form $D(u)$ described above shifted by
eight places. The other four elements are the vectors $c_{11,12},c_{21,22}, c_{31,32}, c_{41,42}$ corresponding
to the binomials $y_{11}-y_{12}, y_{21}-y_{22}, y_{31}-y_{32}, y_{41}-y_{42}$.
}

\end{ex}
\begin{ex}\label{Example B}
{\em 

In this example let $E_2$ be a 16+4 multiset configuration with the ground set the set of columns of $C_1$  where we repeat four  times only the column 12 of $C_1$
(five columns in total identical with the 12-th column)
then we get a 4-pyramidal self dual toric variety according to Theorem \ref{10}. Theorem \ref{all_good} says that the universal
Markov basis and the Graver basis coincide and their elements are described in Theorem \ref{graver}.
They consist of 10 elements in the form $c_{ij,ik}$, where $i=12$ and $j,k$ take different values from 1 to 5.
The rest of the elements are the elements in the sets $\pi^{-1}({\bf u})_+$ where ${\bf u}$ takes the 29 values from the Graver basis of $I_{C_1}$.
Then if the absolute value of the 12-th component of a given vector ${\bf u}$ is $m$ there are \[ f(m) =
\left( \begin{array}{c}
m+4  \\
4
\end{array} \right)
\]
elements in the set $\pi^{-1}({\bf u})_+$, since this is the number of monomials in 5 variables of degree $m$.
There are 8 elements in the Graver basis of $I_{C_1}$ with the $12$-th component $0$, there are 
$8$  $(4, 4, 2, 2, 1)$ elements with the absolute value of the $12$-th component $2024$ $(4048, 6072, 8096, 10120, 14168$ correspondingly).
Therefore there are $$\omega =10 + 8f(0)+ 8f(2024)+4f(4048)+4f(6072)+2f(8096)+2f(10120)+f(14168)$$ elements in the Graver basis $I_{E_2}$, which is also universal Markov by Theorem \ref{all_good}.
From those we need only 33 (=29+4) elements in any minimal Markov basis, four from the 10 in the form $c_{ij,ik}$ and one from each fiber $\pi^{-1}({\bf u})_+$, ${\bf u}\in Gr(C_1)$.

 Let $G$ be a reduced Gr\"obner basis of $I_{E_2}$ with respect to a monomial order $>$.
 Let $y_{12,k}$ be the smallest variable among $y_{12,1}, y_{12,2}, y_{12,3}, y_{12,4}, y_{12,5}$, in the order $>$.
 Then all four elements $y_{12,i}-y_{12,k}$, where $i\not=k$, are in $G$.
 Therefore from each fiber  $\pi^{-1}({\bf u})_+$, such  that ${\bf u}$ has  $u_{12}$ component
 nonzero the only element that belongs to $G$ is the one that the 12th block is only $y_{12,k}^{u_{12}}$.
 Therefore from each fiber like that  $\pi^{-1}({\bf u})_+$ there are exactly five binomials that can appear in the universal  Gr\"obner basis  of $I_{E_2}$,
 in each of these five only one variable from $y_{12,i}$ appears. There are 21 elements  in the Graver basis of $I_{C_1}$ with the 12th $u_{12}$
 component nonzero therefore the cardinality of the universal  Gr\"obner basis is $10+8+5\cdot 21=123$.
 
 There is huge difference between the size of a Markov basis and the universal Markov basis (33 to $\omega$) but also between the size of the universal Gr\"obner basis and the Graver basis (123 to $\omega$).
 In $\omega$ only the $f(14168)$ is about $ 1.68 \cdot 10^{15}.$
 
 Strongly robust ideals have a unique Markov basis while the weakly robust ideals may have a huge number of different minimal Markov bases. Let $\Omega$ denote the number of different minimal Markov bases of $I_{E_2}$. Applying to  $I_{E_2}$ the formula of \cite[Theorem 2.9]{ChKTh} 
  for the number of different minimal Markov bases of a toric ideal $I_A$, we obtain $$\Omega=5^3f(0)^8 f(2024)^8f(4048)^4f(6072)^4f(8096)^2f(10120)^2f(14168).$$
Note that $5^3$ is the number of spanning trees on the complete graph on 5 vertices for the Betti fiber that
has vertices $y_{12,1}, y_{12,2}, y_{12,3}, y_{12,4}, y_{12,5}$. The other Betti fibers have $f(m)+1$ elements, for appropriate $m$, in two connected components, one with $f(m)$ elements  and the other with just one element, see \cite[Section 2]{ChKTh} or \cite[Theorem 1.3.2]{DSS} for details.}
\end{ex}

{\bf Acknowledgements.} The second author has been partially supported by the grant PN-III-P4-ID-PCE-2020-0029, within PNCDI III, financed by Romanian Ministry of Research and Innovation, CNCS - UEFISCDI.

\end{document}